\documentclass[11pt]{article}

\usepackage[utf8]{inputenc}
\usepackage[T1]{fontenc}
\usepackage{lmodern}
\usepackage[english]{babel}
\usepackage{amsmath, amssymb}
\usepackage{amsthm}
\usepackage{enumerate}

\usepackage{titlesec}
\titleformat{\section}{\large\bfseries}{\thesection}{1em}{}

\newtheorem{theorem}{Theorem}
\newtheorem{lemma}{Lemma}
\newtheorem{corollary}{Corollary}
\theoremstyle{remark}
\newtheorem*{remark}{Remark}

\title{Asymptotic Behaviour of the Empirical Distance Covariance for Dependent Data}
\author{Marius Kroll \\ \small Department of Mathematics, Ruhr-Universität Bochum, 44780 Bochum, Germany\\
\small E-Mail: marius.kroll@rub.de}
\date{December 24, 2020}

\begin{document}
\maketitle
\begin{abstract}
We give two asymptotic results for the empirical distance covariance on separable metric spaces without any iid assumption on the samples. In particular, we show the almost sure convergence of the empirical distance covariance for any measure with finite first moments, provided that the samples form a strictly stationary and ergodic process. We further give a result concerning the asymptotic distribution of the empirical distance covariance under the assumption of absolute regularity of the samples and extend these results to certain types of pseudometric spaces. In the process, we derive a general theorem concerning the asymptotic distribution of degenerate V-statistics of order 2 under a strong mixing condition.
\end{abstract}


\paragraph*{MSC2020 subject classification}
62H20 \and 62G20 \and 60F05 \and 30L05

\section{Introduction}
In \cite{lyons}, Lyons introduced the concept of distance covariance for separable metric spaces, generalising the work done by Székely, Rizzo and Bakirov in \cite{srb}. In this very general case, the distance covariance of a measure $\theta$ (on the product space $\mathcal{X} \times \mathcal{Y}$ of separable metric spaces $\mathcal{X}$ and $\mathcal{Y}$) with marginal distributions $\mu$ on $\mathcal{X}$ and $\nu$ on $\mathcal{Y}$ is defined as
\begin{equation*}
\label{eq:dc}
\mathrm{dcov}(\theta) := \int \delta_\theta(z,z') ~\mathrm{d}\theta^2(z,z')
\end{equation*}
for $z = (x,y), z' = (x', y')$, where
\begin{align*}
&\delta_\theta(z,z') := d_\mu(x,x')d_\nu(y,y'), \\
&d_\mu(x,x') := d_\mathcal{X}(x,x') - a_\mu(x) - a_\mu(x') + D(\mu), \\
&a_\mu(x) := \int d_\mathcal{X}(x,x') ~\mathrm{d}\mu(x'), \\
&D(\mu) := \int d_\mathcal{X}(x,x') ~\mathrm{d}\mu^2(x,x').
\end{align*}
To examine the properties of this object, Lyons made use of the concept of (strong) negative type. A metric space $\mathcal{X}$ is said to be of negative type, if there exists a mapping $\phi : \mathcal{X} \to H$ to a Hilbert space $H$, such that $d_\mathcal{X}(x,x') = \|\phi(x) - \phi(x')\|_H^2$ for all $x,x' \in \mathcal{X}$. It is of strong negative type if it is of negative type and $D(\mu_1 - \mu_2) = 0$ if and only if $\mu_1 = \mu_2$ for all probability measures $\mu_1, \mu_2$ with finite first moments. Lyons showed that the distance covariance is non-negative if $\mathcal{X}$ and $\mathcal{Y}$ are of negative type, and that the property $\mathrm{dcov}(\theta) = 0 \Leftrightarrow \theta = \mu \otimes \nu$ holds if $\mathcal{X}$ and $\mathcal{Y}$ are of strong negative type.

This means that the distance covariance completely characterises independence of random variables in metric spaces of strong negative type. Estimators for the distance covariance and their asymptotic behaviour are therefore of great interest for tests of independence.

A special case for real-valued random variables follows from choosing the embedding
\begin{align*}
\phi : \mathbb{R}^d &\to L^2(w_d) := \left\{f : \mathbb{R}^d \to \mathbb{C} ~\Big|~ \int |f|^2 w_d ~\mathrm{d}\lambda^d < \infty\right\} \\
x &\mapsto \frac{1}{\sqrt{2}} (1 - \exp(i\langle .,x\rangle))
\end{align*}
with $w_d(s) = \Gamma((d+1)/2)\pi^{-(d+1)/2}\|s\|_2^{-(d+1)}$, which Lyons in \cite{lyons} refers to as the \textit{Fourier embedding}. This results in the square of the distance covariance as introduced in \cite{srb}, i.e.
$$
\mathrm{dcov}(\theta) = \int |\varphi_{X,Y}(s,t) - \varphi_X(s)\varphi_Y(t)|^2 w_p(s)w_q(t) ~\mathrm{d}(s,t),
$$
where $\varphi_Z$ denotes the characteristic function of a random variable $Z$, and the vector $(X,Y) \in \mathbb{R}^{p+q}$ has distribution $\theta$.

Two of the main results of \cite{lyons} are Proposition 2.6 and Theorem 2.7, which describe the asymptotic behaviour of $\mathrm{dcov}(\theta_n)$, where $\theta_n$ is the empirical measure from $n$ iid-samples of $\theta$. Theorem 2.7, under sufficient moment assumptions, describes the asymptotic distribution of the sequence $n\mathrm{dcov}(\theta_n)$, if $\theta = \mu \otimes \nu$. Proposition 2.6 gives the almost sure convergence $\mathrm{dcov}(\theta_n) \xrightarrow[]{a.s.} \mathrm{dcov}(\theta)$ for any measure $\theta$ with finite first moments. However, as noted by Jakobsen in \cite{jakobsen}, Lyons' proof of Proposition 2.6 was incorrect and actually required $\theta$ to have finite $5/3$-moments. Lyons later acknowledged this in \cite{lyonserrata} (iii), showing that Proposition 2.6 as written in \cite{lyons} is still correct in the case of spaces of negative type, but leaving the question of whether finite first moments are sufficient in the general case of separable metric spaces unanswered. This problem was solved in \cite{janson}, where the almost sure convergence is shown in the case of iid samples.

In Section \ref{sec:metricspaces}, we show that one can obtain the almost sure convergence of the estimator $\mathrm{dcov}(\theta_n)$ under finite first moment assumption while dropping the iid assumption regarding the samples which constitute the empirical measure $\theta_n$. In Theorem \ref{thm:fs}, we show the almost sure convergence of $\mathrm{dcov}(\theta_n)$ under assumption of ergodicity and finite first moments. In Theorem \ref{thm:asymptotik}, we give an asymptotic result similar to Theorem 2.7 in \cite{lyons}, assuming absolute regularity. For this we make use of Theorem \ref{thm:rangzwei}, which is a general result concerning the asymptotic distribution of degenerate V-statistics under the assumption of $\alpha$-mixing data. The definitions of $\alpha$-mixing and absolute regularity are recalled at the end of this section.

A further generalisation can be achieved by raising the metrics of the underlying metric spaces to the $\beta$-th power. We will denote this with $\mathrm{dcov}_\beta$. Typically, $\beta$ is chosen between $0$ and $2$, where the choice $\beta = 1$ results in the regular distance covariance. An equivalent way of describing this is to use the regular definitions of distance covariance, but to consider pseudometric spaces of a particular kind instead of metric spaces, namely those which result from raising some metric to the $\beta$-th power (here, by a pseudometric we refer to a metric for which the triangle inequality need not hold). In Section \ref{sec:pseudometricspaces}, we generalise the results for metric spaces deduced in Section \ref{sec:metricspaces} to pseudometric spaces of this kind.

We now summarise some of the notation used in \cite{lyons}, as well as some basic properties of the distance covariance that will prove useful for our purposes.

Let $X$ and $Y$ be random variables with values in separable metric spaces $\mathcal{X}$ and $\mathcal{Y}$, respectively. We define $Z := (X,Y)$ and write $\theta := \mathcal{L}(Z)$, $\mu := \mathcal{L}(X)$ and $\nu := \mathcal{L}(Y)$, and denote by $\theta_n$ the empirical measure of $Z_1, ..., Z_n$, where $(Z_k)_{k \in \mathbb{N}}$ is a strictly stationary and ergodic sequence with $\mathcal{L}(Z_1) = \theta$. 

If we consider $\mathcal{X}$ to be of negative type via an embedding $\phi$, we denote the Bochner integral $\int \phi ~\mathrm{d}\mu$ with $\beta_\phi(\mu)$, and we write $\hat{\phi}$ for the centered embedding $\phi - \beta_\phi(\mu)$. If $\mathcal{Y}$ is of negative type via $\psi$, we define $\beta_\psi(\nu)$ and $\hat{\psi}$ analogously. If both $\mathcal{X}$ and $\mathcal{Y}$ are of negative type via embeddings $\phi : \mathcal{X} \to H_1$ and $\psi : \mathcal{Y} \to H_2$, we can consider the embedding
\begin{align*}
\phi \otimes \psi : \mathcal{X} \times \mathcal{Y} &\to H_1 \otimes H_2 \\
(x,y) &\mapsto \phi(x) \otimes \psi(y),
\end{align*}
where $H_1 \otimes H_2$ is the tensor product of the Hilbert spaces $H_1$ and $H_2$, equipped with the inner product $\langle u_1 \otimes v_1, u_2 \otimes v_2\rangle_{H_1 \otimes H_2} := \langle u_1, u_2\rangle_{H_1}\langle v_1, v_2\rangle_{H_2}$. 

By Proposition 3.5 in \cite{lyons}, we have that
\begin{equation}
\label{eq:deltatheta}
\delta_\theta(z,z') = 4\langle (\hat{\phi} \otimes \hat{\psi})(z),(\hat{\phi} \otimes \hat{\psi})(z')\rangle_{H_1 \otimes H_2}
\end{equation}
for all $z, z' \in \mathcal{X} \times \mathcal{Y}$, whenever $\mathcal{X}$ and $\mathcal{Y}$ are of negative type via embeddings $\phi$ and $\psi$, respectively.

For the remainder of this paper, we will drop the indices of the metrics on $\mathcal{X}$ and $\mathcal{Y}$ and of the inner products on $H_1$, $H_2$ or $H_1 \otimes H_2$, as it is clear from their arguments which metric or inner product we consider. More precisely, $d$ will denote both a metric on $\mathcal{X}$ and a (possibly different) metric on $\mathcal{Y}$, and $\langle ., .\rangle$ can denote one of three (possibly different) inner products on Hilbert spaces $H_1$, $H_2$ or $H_1 \otimes H_2$.

Recall that for two $\sigma$-algebras $\mathcal{A}$ and $\mathcal{B}$ we define the $\alpha$- and $\beta$-coefficients of $\mathcal{A}$ and $\mathcal{B}$ as
$$
\alpha(\mathcal{A}, \mathcal{B}) := \sup_{A \in \mathcal{A}, B \in \mathcal{B}} \left|\mathbb{P}(A \cap B) - \mathbb{P}(A)\mathbb{P}(B)\right|
$$
and
$$
\beta(\mathcal{A}, \mathcal{B}) := \sup \frac{1}{2} \sum_{i=1}^I\sum_{j=1}^J |\mathbb{P}(A_i \cap B_j) - \mathbb{P}(A_i)\mathbb{P}(B_j)|,
$$
respectively, where the second supremum is taken over all finite partitions $A_1, ..., A_I$ and $B_1, ..., B_J$ such that $A_i \in \mathcal{A}$ and $B_j \in \mathcal{B}$ for all $i$ and $j$.  For a process $(Z_k)_{k \in \mathbb{N}}$, we define
$$
\alpha(n) := \sup_{l \in \mathbb{N}} \alpha(\sigma(Z_1, ..., Z_l), \sigma(Z_{l+n}, Z_{l+n+1}, ...))
$$
and
$$
\beta(n) := \sup_{l \in \mathbb{N}} \beta(\sigma(Z_1, ..., Z_l), \sigma(Z_{l+n}, Z_{l+n+1}, ...)),
$$
and we say that the process $(Z_k)_{k \in \mathbb{N}}$ is $\alpha$-mixing or $\beta$-mixing if $\alpha(n) \xrightarrow[n \to \infty]{} 0$ or $\beta(n) \xrightarrow[n \to \infty]{} 0$, respectively. $\beta$-mixing is also known as absolute regularity. These definitions are taken from \cite{bradley}, where many properties of $\alpha$-mixing and absolutely regular processes are established.

\section{Results for metric spaces}
\label{sec:metricspaces}
We now present our results in the case of separable metric spaces. It should be kept in mind that while we consider the usual distance correlation, Theorems \ref{thm:fs} and \ref{thm:asymptotik} also hold for $\mathrm{dcov}_\beta$ (under appropriate moment conditions). However, we postpone discussion of this until Section \ref{sec:pseudometricspaces}, so as to avoid confusion by abstraction.

The following lemma is a variant of Theorem 3.5 in \cite{billingsley}, where it is formulated for random variables.

\begin{lemma}
\label{lem:gleichgradiglemma}
Let $\mathcal{X}$ be a metrizable topological space, $(\mu_n)_{n \in \mathbb{N}}$ a sequence of measures on $\mathcal{X}$ with weak limit $\mu$ and $h : \mathcal{X} \to \mathbb{R}$ a $\mu$-a.s. continuous function which fulfills the following uniform integrability condition:
\begin{equation}
\label{eq:gleichgradigintbar}
\lim_{M \to \infty} \limsup_{n \to \infty} \int_{\{|h| > M\}} |h| ~\mathrm{d}\mu_n = 0.
\end{equation}
Furthermore, we require $h$ to be dominated by some $\mu$-integrable function $g$, i.e. $|h| \leq g$ $\mu$-a.s. Then $\int h ~\mathrm{d}\mu_n \to \int h ~\mathrm{d}\mu$.
\end{lemma}
\begin{proof}
Without loss of generality, suppose that $\mathcal{X}$ is a metric space. We can decompose the integral with respect to $\mu_n$ into a truncated part and a tail part:
\begin{align*}
\int h ~\mathrm{d}\mu_n &= \int_{\{|h| \leq M\}} h ~\mathrm{d}\mu_n + \int_{\{|h| > M\}} h ~\mathrm{d}\mu_n.
\end{align*}
The truncated integral converges, because it is the integral of an almost surely continuous and bounded function and $\mu_n \Rightarrow \mu$, while the uniform integrability condition \eqref{eq:gleichgradigintbar} implies that the tail integral vanishes in the limit $M, n \to \infty$. More precisely, we have the inequality
\begin{align}
\label{eq:gleichgradigungleichung}
\begin{split}
\limsup_{n \to \infty} \int h ~\mathrm{d}\mu_n &\leq \lim_{M \to \infty} \limsup_{n \to \infty} \int_{\{|h| \leq M\}} h ~\mathrm{d}\mu_n \\
&~~~+ \lim_{M \to \infty} \limsup_{n \to \infty} \int_{\{|h| > M\}} h ~\mathrm{d}\mu_n.
\end{split}
\end{align}
The second summand vanishes by assumption due to \eqref{eq:gleichgradigintbar}. For the first summand, note that for any fixed $M$, the limes superior in $n$ of the integral converges to $\int_{\{|h| \leq M\}} h ~\mathrm{d}\mu$, since $h\textbf{1}_{\{|h| \leq M\}}$ is bounded and almost surely continuous. Furthermore, since $|h\textbf{1}_{\{|h| \leq M\}}| \leq |h| \leq g$, we can employ the dominated convergence theorem to obtain
$$
\lim_{M \to \infty} \limsup_{n \to \infty} \int_{\{|h| \leq M\}} h ~\mathrm{d}\mu_n  = \int h ~\mathrm{d}\mu.
$$
Therefore, the summands in \eqref{eq:gleichgradigungleichung} are indeed well-definded. This gives us
$$
\limsup_{n \to \infty} \int h ~\mathrm{d}\mu_n \leq  \lim_{M \to \infty} \limsup_{n \to \infty} \int_{\{|h| \leq M\}} h ~\mathrm{d}\mu_n + 0 = \int h ~\mathrm{d}\mu.
$$
Since $0 \leq \liminf_{n \to \infty} \int_{\{|h| > M\}} |h|~\mathrm{d}\mu_n \leq \limsup_{n \to \infty} \int_{\{|h| > M\}} |h|~\mathrm{d}\mu_n$ for any $M$, we can use an almost identical argument to obtain
$$
\liminf_{n \to \infty} \int h ~\mathrm{d}\mu_n \geq  \lim_{M \to \infty} \liminf_{n \to \infty} \int_{\{|h| \leq M\}} h ~\mathrm{d}\mu_n + 0 = \int h ~\mathrm{d}\mu,
$$
and thus $\lim_{n\to \infty} \int h~\mathrm{d}\mu_n$ exists and is equal to $\int h~\mathrm{d}\mu$.
\end{proof}

In proving Theorem \ref{thm:fs}, we will make use of the following general result, which is a generalisation of Theorem U (ii) from \cite{aaronson}.

\begin{lemma}
\label{lem:fskonvergenz}
Let $(Z_k)_{k \in \mathbb{N}}$ be a strictly stationary and ergodic process with values in a separable metrizable topological space $\mathcal{Z}$ and marginal distribution $\mathcal{L}(Z_1) = \theta$. Let $h : \mathcal{Z}^d \to \mathbb{R}$ be a measurable function, and let $f : \mathcal{Z} \to \mathbb{R}$ be integrable with respect to $\theta$, so that $|h| \leq f \otimes ... \otimes f$, where the product denoted by $\otimes$ is taken $d$ times and $(f \otimes ... \otimes f)(z_1, ..., z_d) := \prod_{k = 1}^d f(z_k)$. If $h$ is $\theta^d$-a.e. continuous, then $V_h(Z_1, ..., Z_n) \to \int h ~\mathrm{d}\theta^d$ a.s., where $V_h(Z_1, ..., Z_n)$ denotes the $V$-statistics with kernel $h$.
\end{lemma}
\begin{proof}
Without loss of generality, suppose that $\mathcal{Z}$ is a metric space. Let $\theta_n := n^{-1} \sum_{k=1}^n \delta_{Z_k}$ denote the empirical measure of $Z_1, ..., Z_n$. We have the representation $V_h(Z_1, ..., Z_n) = \int h ~\mathrm{d}\theta_n^d$. Furthermore, $\theta_n \Rightarrow \theta$ a.s., since $\mathcal{Z}$ is separable, and therefore $\theta_n^d \Rightarrow \theta^d$ a.s. by Theorem 2.8 (ii) in \cite{billingsley}.

We now wish to employ Lemma \ref{lem:gleichgradiglemma}. Hence, we need to show that the sequence of integrals fulfills the following uniform integrability condition:
$$
\lim_{M \to \infty} \limsup_{n \to \infty} \int_{\{|h| > M\}} |h| ~\mathrm{d}\theta_n^d = 0.
$$

We have
$$
\int_{\{|h| > M\}} |h| ~\mathrm{d}\theta_n^d \leq \int_{\{f \otimes ... \otimes f > M\}} f \otimes ... \otimes f  ~\mathrm{d}\theta_n^d,
$$
and since $\{f \otimes ... \otimes f > M\} \subseteq \bigcup_{i=1}^d M_i$ with $M_i := \{z \in \mathcal{Z}^d ~|~f(z_i) > M^{1/d}\}$, the right hand side is dominated by
$$
\sum_{i=1}^d \int_{M_i} f \otimes ... \otimes f ~\mathrm{d}\theta_n^d = d\left(\int f ~\mathrm{d}\theta_n\right)^{d-1}\int_{\{f > M^{1/d}\}} f ~\mathrm{d}\theta_n,
$$
which, due to Birkhoff's pointwise ergodic theorem, almost surely converges to $d\left(\mathbb{E}_\theta f\right)^{d-1}\mathbb{E}_\theta[\textbf{1}_{\{f > M^{1/d}\}} f]$, where $\textbf{1}_A$ denotes the indicator function of a set $A$. Thus, almost surely,
$$
\lim_{M \to \infty} \limsup_{n \to \infty} \int_{\{|h| > M\}} |h| ~\mathrm{d}\theta_n^d \leq \lim_{M \to \infty} d\left(\mathbb{E}_\theta f\right)^{d-1}\mathbb{E}_\theta[\textbf{1}_{\{f > M^{1/d}\}} f] = 0
$$
since $f$ is assumed to be integrable.

Lemma \ref{lem:gleichgradiglemma} therefore gives us
$$
V_h(Z_1, ..., Z_n) = \int h ~\mathrm{d}\theta_n^d \xrightarrow[n \to \infty]{a.s.} \int h~\mathrm{d}\theta^d.
$$
\end{proof}

Note that the following result does not require any assumptions beyond the separability of the metric spaces $\mathcal{X}$ and $\mathcal{Y}$ and the ergodicity of the samples generating the empirical measure $\theta_n$. Thus, Proposition 2.6 in \cite{lyons} and Theorem 4.4 in \cite{janson}, both of which require iid samples, are consequences of our result. 

\begin{theorem}
\label{thm:fs}
Let $X$ and $Y$ be random variables with values in separable metric spaces $\mathcal{X}$ and $\mathcal{Y}$, respectively, and $Z := (X,Y)$. Write $\theta := \mathcal{L}(Z)$, $\mu := \mathcal{L}(X)$ and $\nu := \mathcal{L}(Y)$, and denote by $\theta_n$ the empirical measure of $Z_1, ..., Z_n$, where $(Z_k)_{k \in \mathbb{N}}$ is a strictly stationary and ergodic sequence with $\mathcal{L}(Z_1) = \theta$. 

If $X$ and $Y$ have finite first moments, i.e. $\mathbb{E}d(X,x_0), \mathbb{E}d(Y,y_0) < \infty$ for some fixed (but arbitrary) $z_0 = (x_0, y_0) \in \mathcal{X} \times \mathcal{Y}$, then
$$
\mathrm{dcov}(\theta_n) \xrightarrow[n \to \infty]{a.s.} \mathrm{dcov}(\theta).
$$
\end{theorem}
\begin{proof}
We follow the idea of the proof of Proposition 2.6 in \cite{lyons}. Consider the symmetric kernel $\bar{h}$, defined as the symmetrisation of $h$, where
$$
h(z_1, ..., z_6) := f(x_1, ..., x_4)f(y_1, y_2, y_5, y_6)
$$
and
$$
f(x_1, ..., x_4) := d(x_1, x_2) - d(x_1, x_3) - d(x_2, x_4) + d(x_3, x_4).
$$

As shown in the proof of Proposition 2.6 in \cite{lyons}, we have
\begin{equation}
\label{eq:habsch}
|h(z_1, ..., z_6)| \leq 4 d(x_2, x_3) d(y_1, y_6).
\end{equation}
Let $z_0 = (x_0, y_0)$ be an arbitrary but fixed point in $\mathcal{X} \times \mathcal{Y}$. Since $a+b \leq ab$ for all real $a, b \geq 2$, we have
$$
d(x,x') \leq d(x,x_0) + d(x',x_0) \leq (2 \lor d(x,x_0))(2 \lor d(x',x_0))
$$
for all $x, x' \in \mathcal{X}$. Now, for $z = (x,y) \in \mathcal{X} \times \mathcal{Y}$, let $\varphi_i(z)$ be defined as $2 \lor d(x,x_0)$ if $i = 2,3$ and as $2 \lor d(y,y_0)$ if $i = 1,6$, and write $\varphi$ for the maximum over all these $\varphi_i$. Using \eqref{eq:habsch}, this gives us
$$
|h(z_1, ..., z_6)| \leq 4 \varphi(z_1)\varphi(z_2)\varphi(z_3)\varphi(z_6).
$$
The functions $\varphi_i$ are continuous and measurable, since the underlying metric spaces are separable. They are also integrable because $X$ and $Y$ are assumed to have finite first moments. Using Lemma \ref{lem:fskonvergenz} therefore gives us $V_{\bar{h}}(Z_1, ..., Z_n) \to \int \bar{h} ~\mathrm{d}\theta^6$ almost surely, where $V_{\bar{h}}(Z_1, ..., Z_n)$ denotes the $V$-statistics with kernel $\bar{h}$. Since the $V$-statistics with kernel $\bar{h}$ are equal to $\mathrm{dcov}(\theta_n)$, and $\int \bar{h} ~\mathrm{d}\theta^6 = \mathrm{dcov}(\theta)$ (cf. \cite{lyons}), this is what we wanted to show.
\end{proof}

\begin{theorem}
\label{thm:rangzwei}
Let $\mathcal{Z}$ be a $\sigma$-compact metrizable topological space, $(Z_k)_{k \in \mathbb{N}}$ a strictly stationary sequence of $\mathcal{Z}$-valued random variables with marginal distribution $\mathcal{L}(Z_1) = \theta$. Consider a continuous, symmetric, degenerate and positive semidefinite kernel $h : \mathcal{Z}^2 \to \mathbb{R}$ with finite $(2+\varepsilon)$-moments with respect to $\theta^2$ and finite $(1+\frac{\varepsilon}{2})$-moments on the diagonal, i.e. $\mathbb{E}|h(Z_1,Z_1)|^{1+\varepsilon/2} < \infty$. Furthermore, let the sequence $(Z_k)_{k \in \mathbb{N}}$ satisfy an $\alpha$-mixing condition such that $\alpha(n) = O(n^{-r})$ for some $r > 1+2\varepsilon^{-1}$. Then, with $V = V_h(Z_1, ..., Z_n)$ denoting the $V$-statistics with kernel $h$,
$$
nV \xrightarrow[n \to \infty]{\mathcal{D}} \sum_{k=1}^\infty \lambda_k \zeta_k^2,
$$
where $(\lambda_k, \varphi_k)$ are pairs of the non-negative eigenvalues and matching eigenfunctions of the integral operator
$$
f \mapsto \int h(.,z)f(z)~\mathrm{d}\theta(z)
$$
and $(\zeta_k)_{k \in \mathbb{N}}$ is a sequence of centered Gaussian random variables whose covariance structure is given by
\begin{equation}
\label{eq:kovarianz}
\mathrm{Cov}(\zeta_i, \zeta_j) = \lim_{n \to \infty} \frac{1}{n} \sum_{t,u=1}^n \mathrm{Cov}(\varphi_i(Z_t), \varphi_j(Z_u)).
\end{equation}
\end{theorem}
\begin{proof}
We note that the conditions of Theorem 2 in \cite{sun} are satisfied by Propositions 1-3 and Assumption 1 ibid., the latter of which is a consequence of $\mathbb{E}|h(Z_1,Z_1)|^{1+\varepsilon/2} < \infty$. Hence we get
$$
h(z,z') = \sum_{k=1}^\infty \lambda_k \varphi_k(z)\varphi_k(z')
$$ 
for all $z,z' \in \mathrm{supp}(\theta)$. The $\varphi_k$ are centered and form an orthonormal basis of $L^2(\theta)$. Adopting the notation $V^{(K)}$ for the $V$-statistics for the truncated kernel $\sum_{k=1}^K \lambda_k \varphi_k(z)\varphi_k(z')$, we note that $nV^{(K)} = \sum_{k=1}^K \lambda_k \zeta_{n,k}^2$, where $\zeta_{n,k} := n^{-1/2} \sum_{t=1}^n \varphi_k(Z_t)$. Using the Cramér-Wold theorem, we will now show that, for any $K \in \mathbb{N}$, $(\zeta_{n,k})_{1 \leq k \leq K}$ weakly converges to $(\zeta_k)_{1 \leq k \leq K}$, where the $\zeta_k$ are centered Gaussian variables with their covariances given in \eqref{eq:kovarianz}. 

Let $c_1, ..., c_K$ be real constants and set $\xi_t := \sum_{k=1}^K c_k\varphi_k(Z_t)$. Then the $\xi_t$ are centered random variables with $\mathbb{E}\xi_t^2 = \sum_{k=1}^K c_k^2$. 

Note that, by definition, $\varphi_k(z) = \lambda_k^{-1}\mathbb{E}[h(z,Z_1)\varphi_k(Z_1)]$ and thus
\begin{equation}
\label{eq:phibetrag}
|\varphi_k(z)| \leq |\lambda_k|^{-1}\|h(z,.)\|_2.
\end{equation}
Here, we have used the Cauchy-Schwarz inequality and the fact that the eigenfunctions $\varphi_k$ form an orthonormal basis of $L^2(\theta)$. This gives us
\begin{align*}
\int |\varphi_k(z)|^{2+\varepsilon} ~\mathrm{d}\theta(z) &\leq \lambda_k^{-(2+\varepsilon)} \int \|h(z,.)\|^{2+\varepsilon} ~\mathrm{d}\theta(z) \\
&= \lambda_k^{-(2+\varepsilon)}  \int \left(\int |h(z,z')|^2~\mathrm{d}\theta(z')\right)^\frac{2+\varepsilon}{2}~\mathrm{d}\theta(z) \\
&\leq \lambda_k^{-(2+\varepsilon)} \int |h(z,z')|^{2+\varepsilon} ~\mathrm{d}\theta^2(z,z')
\end{align*}
by Jensen's inequality, which implies $\|\varphi_k\|_{2+\varepsilon} \leq \lambda_k^{-1}\|h\|_{2+\varepsilon}$. Since our kernel $h$ has finite $(2+\varepsilon)$-moments by assumption, this property translates to the eigenfunctions $\varphi_k$. Using Theorem 3.7 and Remark 1.8 in \cite{bradley} therefore gives us
$$
|\mathrm{Cov}(\varphi_k(Z_t), \varphi_l(Z_u))| \leq C \alpha(\sigma(Z_t), \sigma(Z_u))^{\varepsilon/(2+\varepsilon)} \leq C \alpha(|t-u|)^{\varepsilon/(2+\varepsilon)}
$$
for all $1 \leq k,l \leq K$, where $C$ is a positive constant depending on the corresponding eigenfunctions and -values. From this and the fact that $\alpha(n) = O(n^{-r})$ with $r > 1+2\varepsilon^{-1}$ it follows that, for any $k,l$, the infinite series $\sum_{d=1}^\infty\mathrm{Cov}(\varphi_k(Z_1), \varphi_l(Z_{1+d}))$ and $\lim_n n^{-1}\sum_{d=1}^{n-1} d\mathrm{Cov}(\varphi_k(Z_1), \varphi_l(Z_{1+d}))$ converge, since $d/n < 1$ for all $1 \leq d < n$. Thus, with $S_n$ denoting the sum over $\xi_1, ..., \xi_n$, we have that
\begin{align*}
n^{-1}\sigma_n^2 := n^{-1}\mathbb{E}S_n^2 &= n^{-1} \sum_{t,u=1}^n \sum_{k,l=1}^K c_k c_l \mathrm{Cov}(\varphi_k(Z_t), \varphi_l(Z_u)) \\
&= \sum_{k=1}^K c_k^2 + n^{-1}\sum_{t \neq u}^n \sum_{k,l=1}^K c_k c_l \mathrm{Cov}(\varphi_k(Z_t), \varphi_l(Z_u)) \\
&= \sum_{k=1}^K c_k^2 + n^{-1} 2\sum_{d=1}^{n-1} (n-d) \sum_{k,l=1}^K c_k c_l \mathrm{Cov}(\varphi_k(Z_1), \varphi_l(Z_{1+d})) \\
&\xrightarrow[n \to \infty]{} \sigma^2 < \infty,
\end{align*}
where we have made use of the stationarity of the process $(Z_k)_{k \in \mathbb{N}}$ and the fact that the eigenfunctions $\varphi_k$ form an orthonormal basis of $L^2$. If $\zeta_1, ..., \zeta_K$ are Gaussian random variables with their covariance function given by \eqref{eq:kovarianz}, the limit $\sigma^2$ is the variance of the linear combination $\sum_{k=1}^K c_k\zeta_k$.

We now show the uniform integrability of the sequence $(S_n^2\sigma_n^{-2})_{n\in\mathbb{N}}$. It suffices to show that $\mathbb{E}|S_n\sigma_n^{-1}|^{2+\delta}$ is uniformly bounded in $n$ for some $\delta > 0$. Since $h$ has finite $(2+\varepsilon)$-moments, we get
$$
\sup_{n \in \mathbb{N}}\mathbb{E}\left|\sum_{k=1}^K c_k\varphi_k(Z_n)\right|^{2+\varepsilon} \leq \sup_{n \in \mathbb{N}}\left\{K^{1+\varepsilon} \sum_{k=1}^K \mathbb{E}\left[|c_k\varphi(Z_n)|^{2+\varepsilon}\right]\right\} < M(\varepsilon) < \infty.
$$
Here, we have made use of \eqref{eq:phibetrag} and the stationarity of the sequence $(Z_n)$, which ensures that the upper bound $M(\varepsilon)$ is indeed uniform in $n$. Since $\alpha(n) = O(n^{-r})$ with $r > 1 + 2\varepsilon^{-1} $ and $\sigma_n$ has rate of growth $\theta(\sqrt{n})$, Theorem 2.1 in \cite{sotres} gives us $\mathbb{E}|S_n\sigma_n^{-1}|^{2+\delta} = O(1)$ for some $\delta > 0$. This implies uniform integrability of $(S_n^2\sigma_n^{-2})_{n \in \mathbb{N}}$.

Using Theorem 10.2 from \cite{bradley} therefore gives us
$$
\sum_{k=1}^K c_k \zeta_{n,k} = \frac{S_n}{\sqrt{n}} = \frac{S_n}{\sigma_n} \cdot \frac{\sigma_n}{\sqrt{n}} \xrightarrow[n \to \infty]{\mathcal{D}} \mathcal{N}(0, \sigma^2) = \mathcal{L}\left(\sum_{k=1}^K c_k \zeta_k\right),
$$
and so, by the Cramér-Wold theorem, the vectors $(\zeta_{n,k})_{1 \leq k\leq K}$ converge to Gaussian vectors $(\zeta_k)_{1 \leq k \leq K}$ with the covariance stucture described in \eqref{eq:kovarianz} for any $K \in \mathbb{N}$.

Now, applying the continuous mapping theorem gives us
\begin{equation}
\label{eq:billingsley1}
nV^{(K)} = \sum_{k=1}^K \lambda_k \zeta_{n,k}^{2} \xrightarrow[n \to \infty]{\mathcal{D}} \sum_{k=1}^K \lambda_k \zeta_k^2 =: \zeta^{(K)}
\end{equation}
and the summability of the eigenvalues $\lambda_k$, which is due to the identity $\sum_{k=1}^\infty \lambda_k = \mathbb{E}h(Z_1, Z_1) < \infty$, implies that
\begin{equation}
\label{eq:billingsley2}
\mathbb{E}\left|\zeta- \zeta^{(K)}\right| = \sum_{k>K}\lambda_k \xrightarrow[K \to \infty]{} 0.
\end{equation}

We will now show that
\begin{equation}
\label{eq:billingsley3}
\lim_{K \to \infty}\limsup_{n \to \infty} \mathbb{E}|nV - nV^{(K)}| = 0.
\end{equation}
We consider the Hilbert space $H$ of all real-valued sequences $(a_k)_{k \in \mathbb{N}}$ for which the series $\sum_k \lambda_k a_k^2$ converges, equipped with the inner product given by $\langle (a_k), (b_k)\rangle_H := \sum_k \lambda_k a_k b_k$. Then, writing $T_K(Z_t)$ for the $H$-valued random variable $(0^K, (\varphi_k(Z_t))_{k > K})$, where $0^K$ denotes the $K$-dimensional zero vector, we get
\begin{align*}
\mathbb{E}|nV - nV^{(K)}| &= \mathbb{E}\left[\sum_{k>K} \lambda_k \left(\frac{1}{\sqrt{n}} \sum_{t=1}^n \varphi_k(Z_t)\right)^2\right] \\
&= \mathbb{E}\left\|\frac{1}{\sqrt{n}} \sum_{t=1}^n T_K(Z_t)\right\|_H^2 = \mathrm{Var}\left(\frac{1}{\sqrt{n}} \sum_{t=1}^n T_K(Z_t)\right) \\
&= \frac{1}{n} \sum_{s,t=1}^n \mathrm{Cov}(T_K(Z_s), T_K(Z_t)).
\end{align*}
Here, we define the covariance of two $H$-valued random variables $X$ and $Y$ as the real number $\mathrm{Cov}(X,Y) := \mathbb{E}\langle X,Y\rangle_H - \langle \mathbb{E}X, \mathbb{E}Y\rangle_H$. We aim to employ a covariance inequality for Hilbert-space valued random variables.

For this, let us first consider the $(2+\varepsilon)$-moments of $T_K(Z_1)$. For any $p >0$, we get
\begin{align*}
\|T_K(Z_1)\|_p^p &= \int \|T_K(z)\|_H^p ~\mathrm{d}\theta(z) = \int \left(\sum_{k>K} \lambda_k \varphi_k(z)^2\right)^{p/2}~\mathrm{d}\theta(z) \\
&\leq \int \left(\sum_{k=1}^\infty \lambda_k \varphi_k(z)^2\right)^{p/2}~\mathrm{d}\theta(z) = \int h(z,z)^{p/2} ~\mathrm{d}\theta(z) \\
&= \|h(Z_1, Z_1)\|_{p/2}^{p/2}.
\end{align*}
Since $h$ has finite $(1+\frac{\varepsilon}{2})$-moments on the diagonal by assumption, this implies the $(2+\varepsilon)$-integrability of $T_K(Z_1)$.

Lemma 2.2 in \cite{dehlingvector} and the stationarity of the process $(Z_t)_{t \in \mathbb{N}}$ therefore gives us
$$
|\mathrm{Cov}(T_K(Z_s), T_K(Z_t)| \leq 15 \|T_K(Z_1)\|_{2+\varepsilon}^2 \alpha(|s-t|)^{\varepsilon/(2+\varepsilon)}
$$
and we have shown before that $n^{-1}\sum_{s,t=1}^n \alpha(|s-t|)^{\varepsilon/(2+\varepsilon)}$ converges to a finite limit $c$. Furthermore, from $\|T_K(Z_1)\|_2^2 = \sum_{k>K} \lambda_k \xrightarrow[K \to \infty]{} 0$ and $\|T_K(Z_1)\|_{2+\varepsilon} \leq \|T_1(Z_1)\|_{2+\varepsilon}$ (i.e. the sequence $(T_K(Z_1))_{K \in \mathbb{N}}$ is uniformly $(2+\varepsilon)$-integrable) it follows by Vitali's Theorem that $T_K(Z_1) \xrightarrow[K \to \infty]{(2+\varepsilon)} 0$. Putting all of the above together, we get
\begin{align*}
\lim_{K \to \infty}\limsup_{n \to \infty}\mathbb{E}|nV - nV^{(K)}| &\leq  15c\lim_{K \to \infty}\|T_K(Z_1)\|_{2+\varepsilon}^2 = 0.
\end{align*}
By Theorem 3.2 in \cite{billingsley}, \eqref{eq:billingsley1}, \eqref{eq:billingsley2} and \eqref{eq:billingsley3}, the latter of which we have just shown, imply $nV \xrightarrow[n \to \infty]{\mathcal{D}} \zeta$.
\end{proof}

\begin{lemma}
\label{lem:indizesbeschr}
If $(X_k)_{k \in \mathbb{N}}$ is a strictly stationary sequence of random variables whose marginal distribution $\mu$ has finite $q$-moments, then there exists an upper bound $M \in \mathbb{R}$ such that, for any collection of indices $i_1, ..., i_4$,
$$
\mathbb{E}\left[f(X_{i_1}, ..., X_{i_4})^{2p}\right] \leq M(p) < \infty
$$
for any $p < q$, where $f$ is the function from the proof of Theorem \ref{thm:fs}.
\end{lemma}
\begin{proof}
First, consider any two indices $i_1, i_2$. Then, due to \eqref{eq:weaktriangle}, we have
\begin{align}
\label{eq:momenteabhaengig}
\begin{split}
\mathbb{E}[d(X_{i_1}, X_{i_2})^q] &\leq 2^{q-1} \mathbb{E}[d(X_{i_1}, x_0)^q + d(x_0, X_{i_2})^q \\
&= 2^q \int d(x,x_0)^q ~\mathrm{d}\mu(x) =: M_0 < \infty,
\end{split}
\end{align}
where $x_0$ is some arbitrary point in $\mathcal{X}$.

Now, let $i_1, ..., i_4$ be fixed but arbitrary indices. Then, with a similar bound to the one used in Lemma \ref{lem:pseudosquint},
\begin{align}
\label{eq:momenteabhaengig2}
\begin{split}
&\mathbb{E}[f(X_{i_1}, ..., X_{i_4})^{2p}] \leq 4^p \mathbb{E}[d(X_{i_2}, X_{i_3})^{p} d(X_{i_1}, X_{i_4})^{p}] \\
&\leq 4^{p} \left|\mathbb{E}[d(X_{i_2}, X_{i_3})^{p} d(X_{i_1}, X_{i_4})^{p}] - \mathbb{E}[d(X_{i_2}, X_{i_3})^{p}]\mathbb{E}[d(X_{i_1}, X_{i_4})^{p}]\right| \\
&~~~+ 4^{p} \mathbb{E}[d(X_{i_2}, X_{i_3})^p]\mathbb{E}[d(X_{i_1}, X_{i_4})^p].
\end{split}
\end{align}
We use Lemma 1 from \cite{yoshihara} for the function $h(x_1, ..., x_4) := d(x_1, x_2)^p d(x_3, x_4)^p$ and the reordered collection $(i_2, i_3, i_1, i_4)$. Their assumptions are satisfied with $\delta := \frac{q}{p} - 1$, because
$$
\int h^{1+\delta} ~\mathrm{d}\left(\mathcal{L}(X_{i_2}, X_{i_3}) \otimes \mathcal{L}(X_{i_1}, X_{i_4})\right) = \mathbb{E}[d(X_{i_2}, X_{i_3})^q]\mathbb{E}[d(X_{i_1}, X_{i_4})^q] \leq M_0^2
$$
due to \eqref{eq:momenteabhaengig}. Thus, Lemma 1 in \cite{yoshihara} gives us
\begin{align}
\label{eq:momenteabhaengig3}
\begin{split}
&|\mathbb{E}[d(X_{i_2}, X_{i_3})^{p} d(X_{i_1}, X_{i_4})^{p}] - \mathbb{E}[d(X_{i_2}, X_{i_3})^{p}]\mathbb{E}[d(X_{i_1}, X_{i_4})^{p}]| \\
&\leq 4 M_0^{\frac{2}{1+\delta}} \beta(|i_1 - i_3|)^\frac{\delta}{1 + \delta},
\end{split}
\end{align}
where $\beta(n)$ is the $\beta$-mixing coefficient of the sequence $(Z_k)_{k \in \mathbb{N}}$. Because $\beta(n) \leq 1$ for all $n \in \mathbb{N}$, \eqref{eq:momenteabhaengig}, \eqref{eq:momenteabhaengig2} and \eqref{eq:momenteabhaengig3} give us
$$
\mathbb{E}[f(X_{i_1}, ..., X_{i_4})^{2p}] \leq 4^{p+1}M_0^\frac{2}{1+\delta} + 4^p M_0^2 =: M(p) < \infty.
$$
\end{proof}
The following lemma is an adaptation of Lemma 2 in \cite{yoshihara} in the sense that our result is implicitly contained in their proof. Another variant of this lemma (for U-statistics) can be found in \cite{arcones}. Since both of these lemmas are slightly different from our version, we include a proof for the sake of completeness. However, it should be noted that all three proofs apply the same technique.
\begin{lemma}
\label{lem:varianzdeg}
Let $h$ be a symmetric and degenerate kernel of order $c \geq 2$. Here, we understand degeneracy as $\mathbb{E}h(z_1, ..., z_{c-1}, Z_c) = 0$ almost surely. If, for some $p > 2$, the $p$-th moments of $h(Z_{i_1}, ..., Z_{i_c})$ are uniformly bounded and $(Z_n)_{n \in \mathbb{N}}$ is strictly stationary and absolutely regular with mixing coefficients $\beta(n) = O(n^{-r})$, where $r > cp/(p-2)$, then $\mathbb{E}[V^2] = O(n^{-c})$, where $V = V_h(Z_1, ..., Z_n)$ is the V-statistic with kernel $h$.
\end{lemma}
\begin{proof}
We will follow the basic idea of the proof of Lemma 2 in \cite{yoshihara}. First, consider the special case of $c = 2$. We have
$$
\mathbb{E}\left[\left(\sum_{1 \leq i_1, i_2 \leq n} h(Z_{i_1}, Z_{i_2})\right)^2\right] = \sum_{1 \leq i_1, ..., i_4 \leq n} \mathbb{E}[h(Z_{i_1}, Z_{i_2})h(Z_{i_3}, Z_{i_4})].
$$
Now due to the degeneracy of our kernel $h$, we can employ Lemma 1 in \cite{yoshihara} to obtain
$$
\mathbb{E}[h(Z_{i_1}, Z_{i_2})h(Z_{i_3}, Z_{i_4})] \leq M \cdot \beta\left(\max\{|i_2 - i_1|, |i_4 - i_3|\}\right)^{(p-2)/p}
$$
whenever $(i_1, i_2) \neq (i_3, i_4)$. Here, $M$ is some constant uniform in $i_1, ..., i_4$ and $n$. 

Let us first assume that $k := |i_2 - i_1| \geq |i_4 - i_3| =: l$. For any fixed value of $k$, we have at most $2(n-k)$ possible values for $i_1$. Furthermore, since $k \geq l \geq 0$, we have $k+1$ possible values for $l$ and, for any fixed $l$, at most $2(n-l)$ possible values for $i_3$. Writing
$$\mathcal{I} := \{(i_1, ..., i_4) ~|~ 1 \leq i_1, ..., i_4 \leq n, |i_2 - i_1| \geq |i_4 - i_3|, (i_1, i_2) \neq (i_3, i_4)\}$$
this gives us
\begin{align*}
\sum_{i_1, ..., i_4 \in \mathcal{I}} \mathbb{E}[h(Z_{i_1},Z_{i_2})h(Z_{i_3}, Z_{i_4})] &\leq \sum_{k=0}^{n-1} \sum_{i_1 = 1}^{n-k} \sum_{l=0}^{k} \sum_{i_3 = 1}^{n-l} M \beta(k)^{(p-2)/p} \\
&\leq 4M n^2 \sum_{k=0}^{n-1} (k+1)\beta(k)^{(p-2)/p} \\
&= O(n^2).
\end{align*}
The sum converges due to our assumptions on $\beta(n)$. The same bound can be established for the cases where $|i_4 - i_3| \geq |i_2 - i_1|$. The only combinations missing are those where $(i_1, i_2) = (i_3, i_4)$, of which there are $n^2$. We can combine these results to get
$$
\sum_{1 \leq i_1, ..., i_4 \leq n} \mathbb{E}[h(Z_{i_1}, Z_{i_2})h(Z_{i_3}, Z_{i_4})] = O(n^2),
$$ 
which proves the lemma in the case $c = 2$.

The proof for arbitrary $c$ follows the same idea. We then obtain an upper bound of
$$
2^c M n^c \sum_{k=0}^{n-1} (k+1)^{c-1} \beta(k)^{(p-2)/p} \leq  2^{2c-1} M n^c \sum_{k=0}^{n-1} (k^{c-1} + 1) \beta(k)^{(p-2)/p}
$$
which again is $O(n^c)$ due to our bounds on $\beta(n)$.
\end{proof}

\begin{theorem}
\label{thm:asymptotik}
Let $X$ and $Y$ be random variables with values in separable metric spaces $\mathcal{X}$ and $\mathcal{Y}$, respectively, and $Z := (X,Y)$. Write $\theta := \mathcal{L}(Z)$, $\mu := \mathcal{L}(X)$ and $\nu := \mathcal{L}(Y)$, and denote by $\theta_n$ the empirical measure of $Z_1, ..., Z_n$, where $(Z_k)_{k \in \mathbb{N}}$ is a strictly stationary and ergodic sequence with $\mathcal{L}(Z_1) = \theta$. 

Suppose that $\mathcal{X}$ and $\mathcal{Y}$ are of negative type via mappings $\phi$ and $\psi$, respectively, and that $\mathcal{X} \times \mathcal{Y}$ is $\sigma$-compact. If $X$ and $Y$ are independent, have finite $(1+\varepsilon)$-moments for some $\varepsilon > 0$, and the sequence $(Z_k)_{k \in \mathbb{N}}$ is absolutely regular with mixing coefficients $\beta(n) = O(n^{-r})$ for some $r > 6(1 + 2\varepsilon^{-1})$, then
$$
n\cdot\mathrm{dcov}(\theta_n) \xrightarrow[n \to \infty]{\mathcal{D}} \zeta := \sum_{k=1}^\infty \lambda_k \zeta_k^2 ,
$$
where the $\zeta_k$ are centered Gaussian random variables whose covariance function given in \eqref{eq:kovarianz} is determined by the dependence structure of the sequence $(Z_k)_{k \in \mathbb{N}}$, and the parameters $\lambda_k > 0$ are determined by the underlying distribution $\theta$.
\end{theorem}
\begin{proof}
Consider the identity $\mathrm{dcov}(\theta_n) = V_{\bar{h}}(Z_1, ..., Z_n) =: V$ as given in Theorem \ref{thm:fs}. We will employ Hoeffding decomposition, i.e.
$$
V = \sum_{c=0}^6 {6 \choose c} V_{\bar{h}_c}(Z_1, ..., Z_n),
$$
where
$$
\bar{h}_c(z_1, ..., z_c) = \sum_{A \subset \{1, ..., 6\}} (-1)^{6 - \#A} \int \bar{h}(z_1, ..., z_6) ~\mathrm{d}\theta^{6-c}(z_{c+1}, ..., z_6)
$$
for $0 \leq c \leq 6$. It can be readily seen that under the assumption of independence of $X$ and $Y$, $\bar{h}_1 = 0$ almost surely, and so the Hoeffding decomposition reduces to
\begin{equation}
\label{eq:hoeffding}
V = \sum_{c=2}^6 {6 \choose c} V_{\bar{h}_c}(Z_1, ..., Z_n).
\end{equation}
We will show that the kernel $\bar{h}_2$ satisfies the conditions of Theorem \ref{thm:rangzwei} and that, under our assumptions,
\begin{equation}
\label{eq:probkonvnull}
nV - nV_{\bar{h}_2}(Z_1, ..., Z_n)) \xrightarrow[n \to \infty]{\mathbb{P}} 0.
\end{equation}
Application of some algebra shows that $\bar{h}_2 = \delta_\theta/15$, proceeding in the following way:

It can be easily checked that under independence of $X$ and $Y$, $\bar{h}$ is a degenerate kernel, since integrating over all but one argument of $f$ (with respect to either of the marginal distributions of $\theta$) yields a function which is $0$ almost surely. Therefore, 
$$
\bar{h}_2(z_1, z_2) = \frac{1}{6!}\sum_{\sigma \in \mathfrak{S}_6} \int h(z_{\sigma(1)}, ..., z_{\sigma(6)}) ~\mathrm{d}\theta^4(z_3, ..., z_6),
$$
where $\mathfrak{S}_6$ is the symmetric group of all permutations operating on $\{1, ..., 6\}$. Notice that the summands are equal to $\delta_\theta(z_{\sigma(1)}, z_{\sigma(2)})$ if $\sigma(1), \sigma(2) \in \{1,2\}$. This follows directly from the definitions of $d_\mu$ and $d_\nu$. Moreover, $1$ and $2$ are the only indices appearing in both $f(X_1, ..., X_4)$ and $f(Y_1, Y_2, Y_5, Y_6)$, so any permutation $\sigma$ with $\sigma(1), \sigma(2) \notin \{1,2\}$ results in taking the integral of $f$ over all or all but one argument, either with respect to $\mu$ or with respect to $\nu$. But we have seen before that these integrals are $0$ almost surely, and so, due to the independence of $X$ and $Y$, the same is true for the integral of $h$ with respect to $\theta$.

There are $2\cdot 4!$ permutations of this kind, and so
$$
\bar{h}_2(z_1, z_2) = \frac{2\cdot 4!}{6!}\sum_{\sigma \in \mathfrak{S}_6} \delta_\theta(z_{\sigma(1)}, z_{\sigma(2)}) = \frac{1}{15} \delta_\theta(z_1, z_2).
$$
We can therefore consider the object $\delta_\theta$ instead of $\bar{h}_2$.

By identity \eqref{eq:deltatheta} we have, for any real constants $c_1, ..., c_m$ and $z_1, ..., z_m \in \mathcal{X} \times \mathcal{Y}$,
\begin{align*}
\sum_{i,j=1}^m c_i c_j \delta_\theta(z_i, z_j) &= 4\sum_{i,j=1}^m c_i c_j \langle (\hat{\phi} \otimes \hat{\psi})(z_i), \hat{\phi} \otimes \hat{\psi})(z_j)\rangle \\
&= 4 \left\langle \sum_{i=1}^m c_i (\hat{\phi} \otimes \hat{\psi})(z_i), \sum_{i=1}^m c_i (\hat{\phi} \otimes \hat{\psi})(z_i)\right\rangle\\
&= \left\|2\sum_{i=1}^m c_i (\hat{\phi} \otimes \hat{\psi})(z_i)\right\|^2 \geq 0,
\end{align*}
so our kernel is positive semidefinite. It is furthermore continuous. By Lemma \ref{lem:pseudosquint}, $\delta_\theta$ has finite $(2+\varepsilon)$-moments with respect to $\theta^2$ and finite $(1+\frac{\varepsilon}{2})$-moments on the diagonal. Since $2\alpha(n) \leq \beta(n)$ (cf. \cite{bradley}), we have
\begin{equation}
\label{eq:h2konv}
nV_{\bar{h}_2}(Z_1, ..., Z_n) \xrightarrow[n \to \infty]{\mathcal{D}} \sum_{k=1}^\infty \lambda_k \zeta_k^2
\end{equation}
by Theorem \ref{thm:rangzwei}.

We will now prove \eqref{eq:probkonvnull}. For this, we will first note that under our assumptions, the kernel $\bar{h}$ has finite $(2+\varepsilon)$-moments with respect to $\theta^6$. This can be seen with a similar approach as in the proof of Lemma \ref{lem:pseudosquint}. Furthermore, Lemma \ref{lem:indizesbeschr} together with the independence of $X$ and $Y$ gives us the existence of an upper bound $M \in \mathbb{R}$ such that
$$
\mathbb{E}\left[\bar{h}(Z_{i_1}, ..., Z_{i_6})^{2+\varepsilon}\right] \leq M < \infty
$$
for any collection of indices $1 \leq i_1, ..., i_6 \leq n$.

Employing Lemma \ref{lem:varianzdeg} therefore gives us
$$
\mathbb{E}\left[V_{\bar{h}_c}(Z_1, ..., Z_n)^2\right] = O(n^{-c})
$$
for all $c \geq 2$. Now, together with \eqref{eq:hoeffding}, we have
\begin{align}
\label{eq:2normdiff}
\begin{split}
\mathbb{E}\left[(nV - nV_{\bar{h}_2}(Z_1, ..., Z_n))^2\right] &= \mathbb{E}\left[\left(n\sum_{c=3}^6 {6 \choose c} V_{\bar{h}_c}(Z_1, ..., Z_n)\right)^2\right] \\
&\leq 4n^2 \sum_{c=3}^6 \mathbb{E}\left[V_{\bar{h}_c}(Z_1, ..., Z_n)^2\right] \\
&= \sum_{c=3}^6 O(n^{2-c}) = O(n^{-1}).
\end{split}
\end{align}
This implies \eqref{eq:probkonvnull}, which together with \eqref{eq:h2konv} proves the Theorem.
\end{proof}

Using these two results, we can generalise Corallary 2.8 from \cite{lyons}. 

\begin{corollary}
Under the assumptions of Theorem \ref{thm:asymptotik}, we have
$$
n\frac{\mathrm{dcov}(\theta_n)}{D(\mu_n)D(\nu_n)} \xrightarrow[n \to \infty]{\mathcal{D}} \frac{\sum_{k=1}^\infty\lambda_k \zeta_k^2}{D(\mu)D(\nu)} =:Q
$$
with $\mathbb{E}Q = 1$. If $\mathrm{dcov}(\theta) > 0$, i.e. $\theta$ is not the product measure of its marginal distributions $\mu$ and $\nu$, the left hand side converges to $\infty$ almost surely.
\end{corollary}
\begin{proof}
We have the identity $D(\mu_n) = n^{-2} \sum_{k,l=1}^n d(X_k, X_l)$, and thus by Lemma \ref{lem:fskonvergenz} $D(\mu_n) \xrightarrow{a.s.} D(\mu)$. The same holds for $D(\nu_n)$, and thus the convergence in distribution follows with the Slutsky theorem. Since $D(\mu)D(\nu) = \mathbb{E}\delta_\theta(Z_1, Z_1) = \sum_{k=1}^\infty \lambda_k$, the expected value of the limiting distribution is equal to $1$.

If $\mathrm{dcov}(\theta) > 0$, the almost sure convergence follows by Theorem \ref{thm:fs}.
\end{proof}

\begin{remark}
It would be desirable to achieve a result similar to Theorem \ref{thm:asymptotik} under the assumption of just $\alpha$-mixing. For example, Theorem 3.2 in \cite{davistime} gives such a result under the supposition that $X$ and $Y$ are real-valued random vectors.

For our more general setting of (pseudo-)metric spaces, one only needs to show that \eqref{eq:probkonvnull} still holds in the case of $\alpha$-mixing, since Theorem \ref{thm:rangzwei} does not require absolute regularity. We consider it likely that this can indeed be derived from the amicable properties of the distance covariance.
\end{remark}

\section{Generalisation to pseudometric spaces}
\label{sec:pseudometricspaces}
Let $(\mathcal{X}, d)$ be a metric space and consider $d^\beta$ for $\beta \in (0,2]$. Then $d^\beta$ is a pseudometric, i.e. the triangle inequality does not necessarily hold for $d^\beta$. We will develop parts of the theory of \cite{lyons} for pseudometric spaces of this particular kind, which we will refer to as $\beta$-pseudometric spaces. This is of interest if one considers $\mathrm{dcov}_\beta$, a generalisation of the usual distance covariance, which results from using the $\beta$-th power of the metrics on $\mathcal{X}$ and $\mathcal{Y}$ for the definition of $d_\mu$ and $d_\nu$. That is, $\mathrm{dcov}_\beta$ with respect to $(\mathcal{X}, d)$ and $(\mathcal{Y}, d)$ is equivalent to the regular distance covariance with respect to the $\beta$-pseudometric spaces $(\mathcal{X}, d^\beta)$ and $(\mathcal{Y}, d^\beta)$. Obviously, for any constant $\beta > 0$, $d^\beta$ induces the same topology (and thus, the same Borel $\sigma$-algebra) as the original metric $d$. This means that any $\beta$-pseudometric space is a metrizable topological space.

This approach of viewing $\mathrm{dcov}_\beta$ not as a different object on the same space, but as the same object on a different space might not be very intuitive at first. However, since the concept of (strong) negative type does not require a metric space, this characterisation allows us to still use the relation between (strong) negative type of the underlying space and the distance covariance. This leads to the question of whether $(\mathcal{X}, d^\beta)$ is of (strong) negative type, given the original metric space $(\mathcal{X}, d)$, for which some criteria are known -- see for example Corollary \ref{cor:embedding} or, more generally, \cite{li} and \cite{schoenberg}.

Note that if $\beta \in (0, 1]$, $d^\beta$ is indeed still a metric, and we can rely on the already developed  theory for separable metric spaces. Thus, we get the following result.

\begin{corollary}
Let $\beta \in (0,1]$. Theorems \ref{thm:fs} and \ref{thm:asymptotik} still hold for $\mathrm{dcov}_\beta$ if we replace the finite first moment condition of Theorem \ref{thm:fs} and the finite $(1+\varepsilon)$-moment condition of Theorem \ref{thm:asymptotik} by finite $\beta$- and $(1+\varepsilon)\beta$-moment assumptions, respectively.
\end{corollary}
\begin{proof}
Theorem 1 follows immediately. For Theorem 2, we note that $d^\beta$ induces the same Borel $\sigma$-algebra as $d$. Furthermore, by Remark 3.19 in \cite{lyons}, the resulting metric spaces are still of negative type.
\end{proof}

For $\beta \in (1,2)$, while we cannot rely on the triangle inequality, the Jensen inequality gives us a result which we will call the \textit{weak triangle inequality}. Specifically, for any $\beta \in [1,2]$:
\begin{equation}
\label{eq:weaktriangle}
d^\beta(x, x') \leq 2^{\beta-1} \{d^\beta(x,x_0) + d^\beta(x_0, x')\}
\end{equation}
for all $x, x', x_0 \in \mathcal{X}$. This can be further bounded by replacing the factor $2^{\beta-1}$ by $2$.

Like in the metric case, we say that a probability measure $\mu$ has finite first moment if there exists an element $x_0 \in \mathcal{X}$ such that $\int d(x,x_0) ~\mathrm{d}\mu(x) < \infty$. Again, the choice of $x_0$ is arbitrary due to the weak triangle inequality. Thus, we can define the objects $a_\mu$, $D(\mu)$ and $d_\mu$ as in the metric case.

\begin{lemma}
\label{lem:pseudosquint}
If $\mu$ has finite $\beta p$-moment, then $d_\mu^{(\beta)}$ has finite $2p$-moment with respect to $\mu^2$ and finite $p$-moment on the diagonal for any $p \geq 1$.
\end{lemma}
\begin{proof}
We take inspiration from the proof of Proposition 2.6 in \cite{lyons}. Define the functions
$$
f(x_1, ..., x_4) := d^\beta(x_1, x_2) - d^\beta(x_1, x_3) - d^\beta(x_2, x_4) + d^\beta(x_3, x_4)
$$
and
$$
h(x_1, ..., x_6) := f(x_1, ..., x_4)f(x_1, x_2, x_5, x_6)
$$
We have
\begin{align*}
f(x_1, ... x_4) \leq 2d^\beta(x_1, x_2) - d^\beta(x_1, x_3) - d^\beta(x_2, x_4) + 2d^\beta(x_3, x_4) =: f_+
\end{align*}
and, using the weak triangle inequality, $|f_+| \leq 4d^\beta(x_2,x_3)$. Similarly, we have
$$
f(x_1, ..., x_4) \geq d^\beta(x_1, x_2) - 2d^\beta(x_1, x_3) - 2d^\beta(x_2, x_4) + d^\beta(x_3, x_4) =: f_-.
$$
Again, $|f_-| \leq 4d^\beta(x_2,  x_3)$, and thus $|f(x_1, ..., x_4)| \leq 4d^\beta(x_2, x_3)$. In the same way, one shows that the absolute value of $f(x_1, ..., x_4)$ can also be bounded by $4d^\beta(x_1, x_4)$. Therefore $|h(x_1, ..., x_6)| \leq 16 d^\beta(x_2, x_3)d^\beta(x_1, x_4)$, and so
\begin{align*}
\int |d_\mu^{(\beta)}(x_1, x_2)|^{2p} ~\mathrm{d}\mu^2(x_1, x_2) &= \int\left|\int h(x_1, ..., x_6) ~\mathrm{d}\mu^4(x_3, ..., x_6)\right|^p~\mathrm{d}\mu^2(x_1, x_2) \\
&\leq 16^p\int d^{\beta p}(x_2, x_3)d^{\beta p}(x_1, x_4) ~\mathrm{d}\mu^4(x_1, ..., x_4) \\
&= \left(4^{p/2}\int d^{\beta p}(x,x') ~\mathrm{d}^2(x,x')\right)^2 < \infty.
\end{align*} 
Furthermore, we have
\begin{align*}
\int |d_\mu^{(\beta)}(x,x)|^p ~\mathrm{d}\mu(x) &= \int\left|\int f(x, x, x_3, ..., x_6) ~\mathrm{d}\mu^2(x_3, x_4)\right|^p~\mathrm{d}\mu(x) \\
&\leq 4^p \int d^{\beta p}(x, x_3) ~\mathrm{d}\mu^2(x,x_3) < \infty,
\end{align*}
i.e. $d_\mu^{(\beta)}$ has finite $p$-moment on the diagonal.
\end{proof}
We can now define $\delta_\theta$ and $\mathrm{dcov}(\theta)$ analogously to the metric case. Since the relevant proofs do not make use of the triangle inequality, it follows from \cite{lyons} that for pseudometric spaces of strong negative type $\theta = \mu \otimes \nu$ if and only if $\mathrm{dcov}(\theta) = 0$. This, together with the next Lemma, gives a very easy proof of Theorem 4.2 in \cite{dehlingprozesse}.

\begin{lemma}
\label{lem:hilbertnegativ}
If $(H, \|.\|)$ is a separable Hilbert space, then $(H, \|.\|^\beta)$ is of negative type for all $\beta \in (0,2]$, and of strong negative type for all $\beta \in (0,2)$.
\end{lemma}
\begin{proof}
Without loss of generality, assume $H$ to be equal to $L^2[0,1]$. By Theorem 5 in \cite{schoenberg}, for any $\beta \in (0,2]$, there exists an embedding $\Phi : H \to L^2[0,1]$ with $\|x-x'\|_2^{\beta/2} = \|\Phi(x) - \Phi(x')\|_2$ for all $x, x' \in H$, which implies that $(H, \|.\|^\beta)$ is of negative type. By Remark 3.19 in \cite{lyons} (which, along with all its auxiliary results, also holds for pseudometric spaces), the space $(H, \|.\|^\beta)$ therefore has strong negative type for all $\beta \in (0,2)$.
\end{proof}

We can use this Lemma to adapt Corollary 5.9 from \cite{li}.

\begin{corollary}
\label{cor:embedding}
Let $(\mathcal{X}, d)$ be a metric space. If there exists an isometric embedding from $\mathcal{X}$ into a separable Hilbert space $H$, then $(\mathcal{X}, d^\beta)$ is of negative type for all $\beta \in (0,2]$ and of strong negative type for all $\beta \in (0,2)$.
\end{corollary}
\begin{proof}
Fix $\beta \in (0,2]$, and let $\varphi : \mathcal{X} \to L^2[0,1]$ be an isometric embedding. By Lemma \ref{lem:hilbertnegativ}, $(H, \|.\|_H^\beta)$ is of negative type via some embedding $\Phi$, which implies that $(\mathcal{X}, d^\beta)$ is of negative type via $(\Phi \circ \varphi)$. If $\beta < 2$, then $(H, \|.\|_H^\beta)$ is of strong negative type, and so, for any two probability measures $\mu_1, \mu_2$ on $\mathcal{X}$, we have that
\begin{align*}
D(\mu_1 - \mu_2) &= \int \|\varphi(x) - \varphi(x')\|_H^\beta ~\mathrm{d}(\mu_1^2 - \mu_2^2)(x,x') \\
&=\int_{\varphi(\mathcal{X})^2} \|x-x'\|_H^\beta ~\mathrm{d}\left((\mu_1^\varphi)^2 - (\mu_2^\varphi)^2\right)(x,x') =D(\mu_1^\varphi - \mu_2^\varphi),
\end{align*}
where $\mu_i^\varphi$ denotes the pushforward of $\mu_i$ via $\varphi$. We can extend the last integral to the entire space $H$, because the pushforward measures vanish on $\varphi(\mathcal{X})^C$. Using the strong negative type of $(H, \|.\|_H^\beta)$, this gives us $\mu_1^\varphi = \mu_2^\varphi$, which implies $\mu_1 = \mu_2$, since $\varphi$ is injective.
\end{proof}

\begin{corollary}
Let $\beta \in (1,2)$. Then, if we replace the finite first moment condition of Theorem \ref{thm:fs} by a finite $\beta$-moment assumption, Theorem \ref{thm:fs} still holds for $\mathrm{dcov}_\beta$. If we furthermore assume $\mathcal{X}$ and $\mathcal{Y}$ to be isometrically embeddable into separable Hilbert spaces, and replace the finite $(1+\varepsilon)$-condition with a finite $(1+\varepsilon)\beta$-moment assumption, then Theorem \ref{thm:asymptotik} still holds for $\mathrm{dcov}_\beta$.
\end{corollary}
\begin{proof}
We first consider Theorem \ref{thm:fs}. We can replace \eqref{eq:habsch} by
$$
|h(z_1, ..., z_6)| \leq 16 d^\beta(x_2, x_3)d^\beta(y_1, y_4)
$$
as we have done in the proof of Lemma \ref{lem:pseudosquint}. This changes the original bound only by constant, which does not affect the remainder of the proof.

If $\mathcal{X}$ and $\mathcal{Y}$ are isometrically embeddable into separable Hilbert spaces, then by Corollary \ref{cor:embedding} the spaces resulting from raising their metrics to the power $\beta$ are of negative type. By Lemma \ref{lem:pseudosquint}, the proof of Theorem \ref{thm:asymptotik} still holds for $\beta$-pseudometric spaces. We can therefore apply Theorem \ref{thm:asymptotik} to the spaces $(\mathcal{X}, d^\beta)$ and $(\mathcal{Y}, d^\beta)$.
\end{proof}

\section{Further work}
The limiting distribution established in Theorem \ref{thm:asymptotik} is dependent both on the marginal distribution $\theta$ (through the eigenvaleus $\lambda_k$) and the dependence structure of the process $(Z_k)_{k \in \mathbb{N}}$ (through the Gaussian process $(\zeta_k)_{k \in \mathbb{N}}$). Thus, one cannot directly use this result to construct a test of independence, since the critical values of this test would in general be unknown.

Such a dependence of the limiting distribution on unknown parameters is not unusual -- indeed, in the iid case, there are many well-established ways to approximate the asymptotic distribution of a random variable, even if it may depend on unknown parameters. The authors of \cite{srb}, for instance, propose a permutation test to approximate the asymptotic distribution of the distance covariance for real-valued iid data.

In the case of dependent data, such as we have examined in this paper, one cannot employ methods that would alter the dependence structure of the original sequence $(Z_k)_{k \in \mathbb{N}}$, since this in turn would result in a different Gaussian process $(\zeta_k)_{k \in \mathbb{N}}$ and thus a different limiting distribution. A feasible approach might be a type of block bootstrap (cf. \cite{lahiri}, sections 2.5 -- 2.7), where the resampling occurs from a collection of blocks, each consisting of a certain number of consecutive observations, thus leaving the dependence structure of the original process unchanged. We are currently working on proving the consistency of such a block bootstrap for the distance covariance.

\section*{Acknowledgements}
The author was supported by the German Research Council (DFG) via Research Training Group RTG 2131 (\textit{High dimensional phenomena in probability -- fluctuations and discontinuity}).

\medskip
\bibliographystyle{acm}
\bibliography{dcbib}
\end{document}